\documentclass[11pt]{amsart}
\usepackage{amsmath,amssymb,latexsym,soul,cite,mathrsfs,enumitem}

\usepackage{color,enumitem,graphicx}
\usepackage[colorlinks=true,urlcolor=blue,
citecolor=red,linkcolor=blue,linktocpage,pdfpagelabels,
bookmarksnumbered,bookmarksopen]{hyperref}
\usepackage[english]{babel}

\usepackage[left=2.9cm,right=2.9cm,top=2.8cm,bottom=2.8cm]{geometry}
\usepackage[hyperpageref]{backref}

\usepackage[colorinlistoftodos]{todonotes}
\makeatletter
\providecommand\@dotsep{5}
\def\listtodoname{List of Todos}
\def\listoftodos{\@starttoc{tdo}\listtodoname}
\makeatother

\numberwithin{equation}{section}
\def\dis{\displaystyle}
\def\cal{\mathcal}
\newtheorem{lemma}{Lemma}  
\newtheorem{proposition}{Proposition}
\newtheorem{theorem}{Theorem}
\newtheorem{corollary}{Corollary}
\newtheorem{remark}{Remark}

\DeclareMathOperator*{\essinf}{ess\,inf}
\DeclareMathOperator{\cat}{cat}

\title[Solutions for a SBP system with positive potentials] 
{Multiple  solutions  for a Schr\"odinger-Bopp-Podolsky system
with positive potentials}

\author[G. M. Figueiredo]{Giovany M. Figueiredo}
\author[G. Siciliano]{Gaetano Siciliano}

\address[G. M. Figueiredo]{\newline\indent 
Departamento de Matem\'atica 
\newline\indent 
Universidade de Bras\'ilia - UNB 
\newline\indent 
 CEP: 70910-900 Bras\'ilia DF, Brazil}
\email{\href{mailto:giovany@unb.br}{giovany@unb.br}}

\address[G. Siciliano]{\newline\indent Departamento de Matem\'atica
\newline\indent 
Instituto de Matem\'atica e Estat\'istica
\newline\indent 
 Universidade de S\~ao Paulo 
\newline\indent 
Rua do Mat\~ao 1010,  05508-090 S\~ao Paulo, SP, Brazil }
\email{\href{mailto:sicilian@ime.usp.br}{sicilian@ime.usp.br}}

\thanks{Giovany M. Figueiredo was partially
supported by  CNPq and FAPDF, Brazil. Gaetano Siciliano  was partially supported by
Fapesp grant 2018/17264-4, CNPq grant 304660/2018-3, FAPDF and Capes.}

\begin{document}

\maketitle

\begin{abstract}
In the paper we prove existence of  solutions for a Schr\"odinger-Bopp-Podolsky
system under positive potentials. We use the Ljusternick-Schnirelmann and Morse Theories
to get multiple solutions with {\sl a priori} given ``interaction energy''.

\end{abstract}

\bigskip

\maketitle
\begin{center}
\begin{minipage}{12cm}
\tableofcontents
\end{minipage}
\end{center}

\bigskip

\section{Introduction}

In this paper we are concerned with  existence and multiplicity
results to the following system in $\mathbb R^{3}$
\begin{equation}\label{problema}\tag{$P_{\varepsilon}$}
\left\{
             \begin{array}{l}
              -\varepsilon^{2}\Delta u + Vu+\lambda \phi u+f(u)=0 \medskip\\ 
        -\varepsilon^{2}\Delta \phi + \varepsilon^{4}\Delta^{2} \phi  =u^{2}
             \end{array}
           \right.
\end{equation}
where $\varepsilon > 0$ is a parameter, $V:\mathbb R^{3}\to \mathbb R^{3}$ is a given external potential and 
$f:\mathbb R \to \mathbb R$  
is a  nonlinearity satisfying suitable assumptions that will be given below.
The unknowns are
$$u,\phi:\mathbb R^{3}\to \mathbb R \quad \text{and}\quad \lambda\in \mathbb R.$$

Such a problem has been introduced in \cite{GP} and describes the physical interaction
of a charged particle driven by
the Schr\"odinger equation in the Bopp-Podolsky generalized electrodynamics.
In particular one arrives to a system like \eqref{problema} when looks at
standing waves solutions in the purely electrostatic situation;
indeed $u$ represents the modulus of the wave function of the particle and
$\phi$ is the electrostatic field.
We refer the reader to \cite{GP}
for more details and the physical origin of the system.

Actually there are few papers on Schr\"odinger-Bopp-Podolsky systems.
We cite also \cite{CT,LPT} where the authors study the critical case,
 \cite{H} where the problem has been studied in the Proca setting
on $3$ closed manifolds and
 \cite{GK} where the fibering method of Pohozaev has been 
used to deduce existence of solutions (depending on a parameter)
and even nonexistence.

\medskip
Coming back to our problem,  we  see that,  for any fixed $\varepsilon>0$,
it is equivalent
  to the following one
\begin{equation}\label{equivalente}\tag{$\widetilde{P_{\varepsilon}}$}
\left\{
        \begin{array}{l}
           -\Delta u + V(\varepsilon x)u+\lambda\phi u+f(u)=0 \smallskip
       \\
        -\Delta \phi + \Delta^{2}\phi = u^{2}
             \end{array}
           \right.
\end{equation}
in the sense that, once we find solutions $(\lambda, u, \phi)$ for 
\eqref{equivalente}, 
the triple
 $$\lambda,\quad u(\cdot/\varepsilon), \quad \phi(\cdot/\varepsilon)$$ 
 will be a solution of \eqref{problema}.
  We give now a first set of assumptions. 
  
  On $V$ we start by assuming that \medskip

 \begin{enumerate}[label=(V\arabic*),ref=V\arabic*,start=0]
\item\label{V} $V:\mathbb R^{3}\to \mathbb R$ is in $L^{\infty}_{loc}(\mathbb R^{3})$ 
and satisfies 
$$
0< \essinf_{x\in \mathbb R^3} V(x)=:V_{0}.
$$
\end{enumerate}

The function $f:\mathbb R\to \mathbb R$ is continuous and
\medskip

\begin{enumerate}[label=(f\arabic*),ref=f\arabic*,start=1]
	\item\label{f_{1}} $f(u)\geq0$ for  $u\geq0$ and  $f(u)=0$ for $u\leq0$
\end{enumerate}
\medskip

\noindent or alternatively \medskip
\begin{enumerate}[label=(f\arabic*)',ref=f\arabic*,start=1]
	\item\label{f_{1}'} $f(u)\geq0$ if $u\geq0$ and $f$ is odd, \end{enumerate} 
\medskip
and moreover
\medskip
\begin{enumerate}[label=(f\arabic*),ref=f\arabic*,start=2]
	\item\label{f_{2}} $ \exists \,q\in  (2, 6)$ such that $\lim_{u\rightarrow \infty}{f(u)}/{u^{q-1}}=0$, \medskip
	\item\label{f_{3}} $\lim_{u \rightarrow 0} {f(u)/u}=0$.
\end{enumerate}

\medskip

\noindent  As usual we will denote with $F$ the primitive of $f$ such that $F(0)=0$.

 The natural functional spaces in which find the solutions $u,\phi$ 
of \eqref{equivalente} are
$$
u\in W_{\varepsilon}:=\left\{ u\in H^{1}(\mathbb R^{3}) : \int_{\mathbb R^{3}} V(\varepsilon x)u^{2}<+\infty\right\} $$
$$\phi\in \mathcal D:=\Big\{\phi\in D^{1,2}(\mathbb R^{3}) : \Delta\phi\in L^{2}(\mathbb R^{3})\Big\}
 = \overline{C_{0}^{\infty}(\mathbb R^{3})}^{|\nabla\cdot|_{2}+|\Delta\cdot|_{2}}.
 $$

The space $W_{\varepsilon}$ is an Hilbert space with (squared) norm
$$\|u\|_{W_{\varepsilon}}^{2}:=\int_{\mathbb R^{3}} |\nabla u|^{2} + \int_{\mathbb R^{3}} V(\varepsilon x)u^{2}$$
 and is continuously embedded into $H^{1}(\mathbb R^{3})$.

The space $\mathcal D$ has been 
 introduced and deeply studied in  \cite{GP}, where it is proved that
  $\mathcal D\hookrightarrow L^{p}(\mathbb R^{3})$ for $p\in [6,+\infty].$

Actually  problem \eqref{equivalente} can be simplified more.
Indeed, as it is standard in these kind of systems (see  \cite{GP} for details) 
 a usual {\sl reduction argument} transforms  \eqref{equivalente} into the following nonlocal equation 
\begin{equation}\label{eq:equazione}
-\Delta u+ V(\varepsilon x) u + \lambda\phi_{u} u +f(u)=0 \quad \text{ in } \ \mathbb R^{3}
\end{equation}
where 
\begin{equation}\label{eq:phiu}
\phi_{u}(x)=\int_{\mathbb R^{3}} \frac{1-e^{-|x-y|}}{|x-y|}u^{2}(y)dy.
\end{equation}
Moreover $\phi_{u}\in \mathcal D$ if $u\in H^{1}(\mathbb R^{3})$.
Hence from now on we will refer always to \eqref{eq:equazione}
in the  only unknowns $u$ and $\lambda$, since $\phi_{u}$ is  
determined by $u$
by the above formula. 

\medskip

Fixed $ \varepsilon>0 $, by a solution of \eqref{eq:equazione} we mean a pair 
$(u,\lambda)\in   W_{\varepsilon}\times \mathbb R$ 
such that 
\begin{equation}\label{eq:defsol}
\int_{\mathbb R^{3}}  \nabla u \nabla v 
+\int_{\mathbb R^{3}}  V(\varepsilon x) u v+\lambda\int_{\mathbb R^{3}}  \phi_{u}
u v+\int_{\mathbb R^{3}} f(u) v = 0, \quad \forall v\in W_{\varepsilon}. 
\end{equation}
Note that under our assumptions, all the integrals appearing in \eqref{eq:defsol} are finite
and  the relation between $\lambda$ and $u$ is given,  for $u\not\equiv0$, by
 $$\lambda=
  -\frac{\|u\|_{W_{\varepsilon}}^{2} +\displaystyle\int_{\mathbb R^{3}} f(u)u}{\displaystyle\int_{\mathbb R^{3}} \phi_{u}u^{2}}
 $$
 and so in particular $\lambda$ is negative.
 
 It is also clear that $(0,\lambda)$, $\lambda\in \mathbb R$, is a solution of \eqref{eq:equazione}, that we call {\sl trivial}.
 Of course we are interested in nontrivial solutions namely solutions with $u\not\equiv0$.

Our next  assumption is  

\begin{enumerate}[label=(C),ref=C]
\item\label{C} $W_{\varepsilon}\hookrightarrow \hookrightarrow L^{p}(\mathbb R^{3})$ for $p\in (2,6)$.
\end{enumerate}
The compact embedding  can be achieved in various ways. For example,
\begin{itemize}
\item
by imposing that $V$ is coercive.
In this case it is known that $W_{\varepsilon}$ has compact embedding into $L^{p}(\mathbb R^{3}),p\in [2,6)$;
\item by imposing that for any $c,r>0$
$$\text{meas}\left\{ x\in B_{r}(y): V(x)\leq c \right\}\to 0 \quad \text{as } \ |y|\to+\infty.$$
 Hereafter
$B_{r}(y)$ is the ball in $\mathbb R^{3}$
 with radius $r>0$ centred in $y$.
Also in this case
the embedding is compact  into $L^{p}(\mathbb R^{3}),p\in [2,6)$,
see \cite[pag. 553]{BPW};
\item
by imposing that $V$ is radial. In this case the natural setting to work with is the radial framework,
namely  the subspace of radial functions in $W_{\varepsilon}$
(if $u$ is radial, also $\phi_{u}$ is) which has compact embedding into $L^{p}(\mathbb R^{3}), p\in(2,6)$. 
This setting is justified by the Palais' Principle of Symmetric 
Criticality
and then the solutions found will satisfy \eqref{eq:defsol} even when tested on nonradial functions of $W_{\varepsilon}$.
Then, if $V$ is radial, all the solutions $u$ found in the theorems below 
are radial too.
\end{itemize}

\medskip

In the following we will simply speak of ``negative, one sign or sign-changing solutions''
to say that $u$ is negative, one sign  or sign-changing.

The solutions $(u,\lambda)$ of \eqref{eq:equazione} will be found as critical points 
of a $C^{1}$  energy functional $I_{\varepsilon}$ restricted to the surface energy
(known in physics also as Fermi surface)
$$\left\{u\in W_{\varepsilon}: \int_{\mathbb R^{3}} \phi_{u} u^{2} = 1\right\}$$
and then $\lambda$ will be the associated Lagrange multiplier.
In this context,
using a standard terminology, we mean by a {\sl ground state solution} 
a solution $u$  whose energy $I_{\varepsilon}(u)$ is minimal
(on the constraint) among all the solutions.

\medskip

The results proved here are of two type, depending essentially if \eqref{f_{1}}
or \eqref{f_{1}}' is assumed.

We start with  the assumption \eqref{f_{1}}\,\hspace{-0,1cm}'.
In this case infinitely many solutions with divergent energy are found and 
they are possible sing-changing.

\begin{theorem}\label{th:signchanging}
Assume \eqref{f_{1}}\,\hspace{-0,1cm}', 
\eqref{f_{2}}, \eqref{f_{3}}, \eqref{V} and \eqref{C}. Then, for any $\varepsilon>0$,
 problem \eqref{eq:equazione}
possesses infinitely many  solutions $( u_{n}, \lambda_{n})$ with
\begin{eqnarray*}
&&\|u_{n}\|_{W_{\varepsilon}} \to +\infty, \qquad \frac{1}{2}\|u_{n}\|_{W_{\varepsilon}}^{2}+\int_{\mathbb R^{3}}F(u_{n})\to +\infty \\
&&
\qquad \lambda_{n} = -\left(\|u_{n}\|_{W_{\varepsilon}}^{2} +\int_{\mathbb R^{3}} f(u_{n})u_{n}\right)\to-\infty.
\end{eqnarray*} 
The ground state solutions can be assumed of one sign.
\end{theorem}

The next three  theorems deal with the existence of  solutions $(u,\lambda)$
under assumption \eqref{f_{1}}. 

We state explicitly a first result on the existence of ground state.
\begin{theorem}\label{th:negativo}
Assume \eqref{f_{1}}-\eqref{f_{3}}, \eqref{V} and \eqref{C}. Then, for any $\varepsilon>0$,
 problem \eqref{eq:equazione} admits a ground state solution which is negative.
\end{theorem}
To get multiplicity results it will be important the smallness of $\varepsilon$
and the  topological properties 
of the set of minima of the potential $V$, when achieved.
 Then our next assumption stronger then \eqref{V}:

\medskip

 \begin{enumerate}[label=(V\arabic*),ref=V\arabic*,start=1]
\item\label{V1} $V:\mathbb R^{3}\to \mathbb R$ is continuous 
and satisfies 
$$
0<\min_{x\in \mathbb R^3}V(x)=:V_{0}, \quad \text{with }\ 
M:=\Big\{x\in\mathbb R^{3}:V(x)=V_{0}\Big\} \ni 0.
$$
\end{enumerate}

\medskip

Recall that $\cat_{Y}(X)$ denotes the Ljusternick-Schnirelmann category
of the set $X$ in $Y$; that is, it is the least number of closed and contractible sets
in $Y$ which cover $X$.  If $X=Y$ we just write $\cat(X)$.
%

\begin{theorem}\label{th:LST}
Assume \eqref{f_{1}}-\eqref{f_{3}}, \eqref{V1} and \eqref{C}.
Then, there exists $\varepsilon^{*}>0$ 
such that for every $\varepsilon\in (0,\varepsilon^{*}]$
problem \eqref{eq:equazione} has at least
 $\cat(M)$ negative solutions with low energy.

If moreover if $M$ is bounded and $\cat (M)>1$
there is another negative solution with high energy.

\end{theorem}

The meaning of ``low energy'' or ``high energy'' will be clear during 
the proof.

A second multplicity result of negative solutions is obtained by making use of the Morse Theory.
In this case 
we introduce the next set of assumptions stronger then the previous one on $f$:
\medskip

\begin{enumerate}[label=(f\arabic*),ref=f\arabic*,start=4]
\item\label{fC1}  $f$ is  $C^{1}$, $f(u)\geq0$ for $u\geq0$   and $f(u)=0$ for $u\leq0$; \medskip
\item\label{f_{4}} $ \exists \,q\in  (2, 6)$ such that $\lim_{u\rightarrow \infty}{f'(u)}/{u^{q-2}}=0$; \medskip
\item\label{f_{5}} $ \lim_{u\rightarrow 0}f'(u)=0$.
\end{enumerate}

\medskip

In the following  $\mathcal P_{t}(M)$ is the Poincar\'e polynomial of $M$.

\begin{theorem}\label{th:MT}
Assume \eqref{fC1}-\eqref{f_{5}}, \eqref{V1} and \eqref{C}. Then there exists $\varepsilon^{*}>0$ such that for every $\varepsilon\in (0,\varepsilon^{*}]$ problem \eqref{eq:equazione}
 has at least $2\mathcal P_{1}(M)-1$ negative 
 solutions, possibly counted with their multiplicity.
\end{theorem}

\medskip

It is clear that in general, we get a better result using the Morse Theory.
For example, 
\begin{itemize}
\item 
if $M$ is obtained by a contractible domain cutting off $k$
 disjoint contractible sets, it is 
 $$\cat(M) = 2, \quad \text{and} \quad \mathcal P_{1}(M) = 1 + k;$$
\item 
if $M$ is obtained as a  union of $l$ spheres $\{S_{i}\}_{i=1,\ldots l}$ and $m$ anuli $\{A_{j}\}_{j=1,\ldots ,m}$ 
all pairwise disjoint, 
 then, since $\cat (S_{i}) = \cat(A_{j}) = \mathcal P_{1}(S_{i}) = \mathcal P_{1}(A_{j})=2$ we get
$$\cat(M)=2(l+m) \quad \text{and} \quad 2\mathcal P_{1}(M)-1=2\cdot 2(l+m) -1. $$
\end{itemize}

\medskip

As we said above, our approach is variational. In particular, to prove  Theorem \ref{th:LST}
 and Theorem \ref{th:MT} a fundamental role is played by the autonomous problem
\begin{equation}\label{eq:autonomo}
-\Delta u + V_{0} u +\lambda \phi_{u} u +f(u) = 0\quad \text{ in }\mathbb R^{3},
\end{equation}
and especially by  its ground state solution $\mathfrak u$,
that is, the minimum of the associated energy functional (denoted with $E_{V_{0}}$) 
on the functions $u\in H^{1}(\mathbb R^{3})$ satisfying
$$ \int_{\mathbb R^{3}} \phi_{u} u^{2} = 1.$$
{\sl En passant} we then prove existence and multiplicity  results for \eqref{eq:autonomo}, see Theorem \ref{th:gs-mu-infty}
and Theorem \ref{th:gs-mu} in Section \ref{sec:autonomous}.

\begin{remark}\label{rem:c}
	Observe finally that all the solutions we find satisfy 
	$$\int_{\mathbb R^{3}} \phi_{u} u^{2}=1$$
	but indeed  the results  are  evenly true if we consider solutions with
$$\int_{\mathbb R^{3}} \phi_{u} u^{2}=c, \quad c>0.$$	
\end{remark}

\begin{remark}
Our theorems  are true also if the potential $f$ depends explicitly
	on $x\in \mathbb R^3$. In this case the limits in \eqref{f_{2}}, \eqref{f_{3}}, \eqref{f_{4}},
	\eqref{f_{5}} have to be uniform in $x$. In this case, some degeneracy in $x$ is also permitted, in the sense that $f$
	can be zero for $x$ in some region $\mathcal R$ of $\mathbb R^{3}$.
	Physically speaking it means that the potential $f$ is acting only on $\mathbb R^3\setminus\mathcal R$.
	\end{remark}

\medskip

Let us briefly comment now our assumptions.

First of all observe that, under  \eqref{f_{1}}
or \eqref{f_{1}}', if $\lambda\geq0$ is given {\sl a priori}
we do not have any nontrivial solution.
Indeed if $u$ is a solution of \eqref{eq:equazione}, just multiplying the equation by the same
$u$ and integrating, we reach $u\equiv 0$.
Moreover 
the positivity of $f$ in case \eqref{f_{1}} will be important in proving that the ground state
solution of the autonomous problem \eqref{eq:autonomo}  is radial. Note that the constraint on which we will restrict $E_{V_{0}}$
is not closed under the radial decreasing  rearrangements.

Assumption \eqref{f_{2}} and \eqref{f_{3}}
are standard when using variational methods: they will allow to define a $C^{1}$ energy
functional related to the problem. Analogously, the stronger assumptions \eqref{fC1}-\eqref{f_{5}}
will be useful to deal with the second derivative of the functional and implementing the Morse Theory.

In particular our assumptions on $f$ cover the case $f\equiv0$.

As we have seen,  assumption \eqref{V} is useful to define the right functional
spaces and \eqref{V1} will be useful to deal with the multiplicity result via the category of Ljusternick and Schnirelmann.

Finally assumption \eqref{C} will be important in order to recover the compactness condition
of Palais and Smale, recalled in Section \ref{sec:preliminari}.

 \medskip

To prove the  result we will be mainly inspired by the classical papers \cite{BC1,BCP,BC}
where a general method  to obtain multiplicity of solutions depending on the topology of the ``domain''
has been developed.
Later on, many other problems (involving quasilinear or fractional  equations, among many others)
 have been treated with the same ideas: 
we just recall here  \cite{Alves,Alvesgio1,claudianorserginhorodrigo,CGU,CLV, CV2, CV3,
	FPS, GhiMic, Sicilia,Visetti}.
 However there are evident differences with our paper.

In these last  cited papers the functional is unbounded below on the space
and the constraint is the well known Nehari manifold. The advantage of working
on the Nehari manifold is that the functional becomes bounded below.
Moreover this constraint is  introduced as the set of zeroes of a function which involves the same energy functional
(actually its derivative) and is a natural constraint. In this way suitable conditions on the nonlinearity $f$ (e.g. Ambrosetti-Rabinowitz type condition) 
permits to obtain the boundedness of the Palais-Smale sequence and then compactness results. 
We recall that in this cases, an additional assumption on $V$ is set at infinity:
$$V_{0}< \liminf_{|x|\to +\infty}V(x)=:V_{\infty}\leq +\infty$$
which is  useful to obtain compactness.
Moreover when dealing with the constraint of the Nehari manifold a great help is given by the fact
that there is a minimax characterization  of the projection of any nonzero element 
on the constraint.

In our case the functional is positive  
on the whole space and the  constraint has nothing to do with the functional.
In particular it is always possible to project nonzero functions on the constraint
and we do not require  assumptions of $f(u)/u$.
Moreover,  although we have again the uniqueness of the projection, the minimax characterisation is lost. 
Observe that  we do not need any Ambrosetti-Rabinowitz type condition
and even more, our nonlinearity can vanish somewhere.
For these reasons, although we follow the general strategy
of the cited papers,   many classical proofs do not work and need  to be readjusted.
Another difference from the classical papers, is that our solutions of Theorem
\ref{th:LST} and Theorem \ref{th:MT} are negative.


Then, to the best of our knowledge this is the first paper dealing with the ``photography method''
of Benci, Cerami and Passaseo
with assumptions on the nonlinearity different from the usual ones.

\medskip

We believe that an interesting problem will be the study of multiplicity
of solutions in other cases in which $f$ is negative, as well as,
remove assumption $\eqref{C}$ and address  the problem 
with the approach of Lions by using concentration compactness arguments.

\medskip

The organization of the paper is the following. 

In Section \ref{sec:preliminari},
 we introduce  the related variational setting, the constraint and its fundamental properties,
and we show the compactness property of the functional.

In the brief Section \ref{sec:genus}, we prove
 Theorem \ref{th:signchanging} and Theorem \ref{th:negativo}.

In Section \ref{sec:autonomous} we study the autonomous problem \eqref{eq:autonomo},
with a general positive constant  $\mu$ instead of $V_{0}$. 
We obtain  multiplicity results  of infinitely many solutions if \eqref{f_{1}}' holds
(see Theorem \ref{th:gs-mu-infty}). 
In case \eqref{f_{1}} holds, we found the important 
result concerning the ground state solution (see Theorem \ref{th:gs-mu})
that will be used later on to implement the barycentre machinery.

In the final Section \ref{Bary},  after defining the barycentre maps and its properties,
 we prove Theorem \ref{th:LST} and Theorem \ref{th:MT}.

\medskip

{\bf Notations.}
Here we list few notations that will be used through the paper.
Other will be introduced whenever we need.
\begin{itemize}
\item $|\cdot |_{p}$ is the $L^{p}-$norm;
\item  $H^{1}(\mathbb R^{3})$ is the usual Sobolev space with norm $\|\cdot\|$;
\item \textcolor{red}{given a set of functions $X$, the subset of radial functions is denoted with $X_{r}$};
\item the conjugate exponent of $r$ is denoted with $r'$;
\item  $o_{n}(1)$ denotes a vanishing sequence;
\item  $C,C',\ldots$ stand to denote suitable positive constants whose values may 
also change from line to line.
\end{itemize}


\section{Preliminaries and variational setting}\label{sec:preliminari}


%
%
Let us start with few preliminaries and recalling some well-known facts.

It is standard that from the growth conditions on $f$ and $f'$ given in \eqref{f_{2}}, \eqref{f_{3}},
\eqref{f_{4}} and \eqref{f_{5}}, it follows that for any 
$\delta>0$ there exists $C_{\delta}>0$ such that 
 for every $v,w\in H^{1}(\mathbb R^{3})$,
\begin{equation}\label{eq:limitaf}
\int_{\mathbb R^{3}} |f(u)v| \leq\delta \int_{\mathbb R^{3}} |u v| +C_{\delta}
\int_{\mathbb R^{3}} |u|^{q -1} |v|
\leq \delta |u|_{2} |v|_{2}+C_{\delta} |u|^{q-1}_{q} |v|_{q}.
\end{equation}
and
\begin{equation}\label{eq:limitaf'}
\int_{\mathbb R^{3}} |f'(u)vw| \leq\delta \int_{\mathbb R^{3}} | v w| +C_{\delta}
\int_{\mathbb R^{3}} |u|^{q -2} |v||w|
\leq \delta |v|_{2} |w|_{2}+C_{\delta} |u|^{q-2}_{q} |v|_{q}|w|_{q}.
\end{equation}

For completeness we recall the following properties of $\phi_{u}$ defined in \eqref{eq:phiu}. 
They are contained in \cite[Lemma 3.4 and Lemma 5.1]{GP}.
 \begin{lemma}\label{lem:phi}
	For every $u\in H^{1}(\mathbb R^{3})$ we have:
	\begin{enumerate}[label=(\roman*),ref=\roman*]
		\item\label{propphiii} for every $y\in\mathbb R^3$, $\phi_{u( \cdot+y)} = \phi_{u}( \cdot+y)$; \smallskip
		\item\label{propphiiii} $\phi_{u}\geq0$; \smallskip
		\item \label{propphiiv} for every  $s\in (3,+\infty]$, $\phi_{u}\in   L^{s}(\mathbb R^{3})\cap C_{0}(\mathbb 
		R^{3})$; \smallskip
		\item \label{propphiv} for every $s\in (3/2,+\infty]$, $\nabla \phi_{u} = \nabla \mathcal K * u^{2}\in L^{s}
		(\mathbb R^{3})\cap C_{0}(\mathbb R^{3})$; \smallskip
		\item \label{propphivii} $|\phi_{u}|_{6}\leq C \|u\|^{2}$  for some  constant $C>0$; \smallskip
		\item \label{propphiviii} $\phi_{u}$ is the unique minimizer of the functional
		$$E(\phi) = \frac12 |\nabla \phi|_{2}^{2} +\frac {1}{2} |\Delta\phi|_{2}^{2}- \int_{\mathbb R^{3}}\phi u^{2}, \quad \phi\in \mathcal D.$$
			\end{enumerate}
Moreover if $u$ is radial also $\phi_{u}$ is and if
 $u_n \rightharpoonup u$ in $H_r^1(\mathbb{R}^3)$, then
	\begin{enumerate}[label=(\roman*),ref=\roman*]\setcounter{enumi}{6}
\item\label{propphiiii}   $\phi_{u_{n}} \to \phi_u$  in $\mathcal{D}$; \smallskip
\item\label{propphivv}$\displaystyle\int_{\mathbb R^{3}} \phi_{u_{n}} u_{n}^{2} \to \int_{\mathbb R^{3}} \phi_{u} u^{2}$; \smallskip
\item\label{propphivi} $\displaystyle\int_{\mathbb R^{3}} \phi_{u_{n}} u_{n}v \to \displaystyle\int_{\mathbb R^{3}}\phi_{u} u v$ for any $v\in 
H^{1}(\mathbb R^{3})$.
	\end{enumerate}
 \end{lemma}


Let us  recall finally the following Hardy-Littlewood-Sobolev Inequality,
see e.g. \cite[Theorem 4.3]{LL}.
\begin{theorem}\label{HLS}
	Assume that $1<a,b<\infty$ satisfies
	\begin{equation*}
	\frac{1}{a}+\frac{1}{b}=\frac{5}{3}.
	\end{equation*}
	Then there exists a constant $H>0$ such that 
	\begin{equation*}
	\left|\int_{\mathbb R^{3}}\int_{\mathbb{R}^3}\frac{f(x)g(y)}{|x-y|} \right|\le H |f|_a |g|_b, \quad \forall f\in L^a(\mathbb{R}^3), g\in L^b(\mathbb{R}^3).
	\end{equation*}
\end{theorem}
As a consequence   we get
\begin{proposition}\label{prop:weaklyclosed}
Under assumption \eqref{V}, if $\{u_{n},u\}\subset W_{\varepsilon}$ is such that $u_{n}\to u$ in $L^{12/5}(\mathbb R^{3})$, then
$$\int_{\mathbb R^{3}} \phi_{u_{n}} u_{n}^{2} \to \int_{\mathbb R^{3}} \phi_{u} u^{2}.$$
\end{proposition}
\begin{proof}
Indeed, by the Hardy-Littlewood-Sobolev Inequality,
\begin{eqnarray*}
\int_{\mathbb R^{3}} \left| \phi_{u_{n}} u_{n}^{2} - \phi_{u} u^{2}\right|
&=&\int_{\mathbb R^{3}}\int_{\mathbb R^{3}}  \frac{1-e^{-|x-y|}}{|x-y|} \left| u^{2}_{n}(x)u_{n}^{2}(y) - u^{2}(x)u^{2}(y)\right| \\
&\leq &\int_{\mathbb R^{3}}\int_{\mathbb R^{3}}  \frac{1}{|x-y|} \left| u^{2}_{n}(x)u_{n}^{2}(y) - u^{2}(x)u^{2}(y)\right| \\
&\leq &\int_{\mathbb R^{3}}\int_{\mathbb R^{3}}  \frac{ \left| u_{n}^{2}(y) - u^{2}(y)\right|  }{|x-y|} u_{n}^{2}(x)
+\int_{\mathbb R^{3}}\int_{\mathbb R^{3}}  \frac{ \left| u_{n}^{2}(x) - u^{2}(x)\right|  }{|x-y|} u^{2}(y)\\
&\leq& H \left| u_{n}^{2} - u^{2}\right|_{6/5}  \left(|u_{n}^{2}|_{6/5} + |u^{2}|_{6/5} \right) \\
&=&o_{n}(1)
\end{eqnarray*}
and the conclusion follows.
\end{proof}

The strategy to find solutions $(u_{\varepsilon},\lambda_{\varepsilon})\in   W_{\varepsilon}
\times \mathbb R$ for \eqref{eq:equazione}
will be to look at the critical points  of the functional
\begin{equation}\label{eq:functional}
I_{\varepsilon}(u) = \frac{1}{2}\int_{\mathbb R^{3}}|\nabla u|^{2} + \int_{\mathbb R^{3}} V(\varepsilon x) u^{2} 
+ \int_{\mathbb R^{3}}F(u) = \frac{1}{2}\|u\|_{W_{\varepsilon}}^{2} + \int_{\mathbb R^{3}}F(u) 
\end{equation}
restricted to the set
\begin{equation*}
\mathcal M_{\varepsilon}= \left\{ u \in W_{\varepsilon}:
J(u) =0\right\}, \quad \text{ where } \  J(u):= \int_{\mathbb R^{3}} \phi_{u} u^{2} - 1.
\end{equation*}
Observe that $\mathcal M_{\varepsilon}\neq\emptyset$. Indeed,
fix $u\neq0$ and define 
$$h:t\in(0,+\infty)\to \mathbb R \quad  \text{such that} \quad h(t):  = t^{4} \int_{\mathbb R^{3}}\phi_{u}u.$$
Then there is  a unique positive value $t_{\varepsilon}(u)>0$ such that 
\begin{equation}\label{eq:te}
1 = t_{\varepsilon}(u)^{4}\int \phi_{u} u^{2} = \int \phi_{t_{\varepsilon}(u) u} (t_{\varepsilon}(u) u)^{2}, \quad \text{i.e. } 
t_{\varepsilon}(u)u\in \mathcal M_{\varepsilon}.
\end{equation}
Of course the value $ t_{\varepsilon} $ does not have any minimax caracterization
as it happens with the Nehari constraint.
Note that $t_{\varepsilon}(u) = t_{\varepsilon}(-u)$ and 
it is clear that 
\begin{eqnarray*}
&&\forall \varepsilon_{1}, \varepsilon_{2}>0  : \  \mathcal M_{\varepsilon_{1}} = \mathcal M_{\varepsilon_{2}},\\
&& u\in \mathcal M_{\varepsilon} \Longrightarrow \pm|u| \in \mathcal M_{\varepsilon}.
\end{eqnarray*}
Moreover we have immediately the following
\begin{lemma}\label{lem:noLp}
If $\{u_{n}\}\subset \mathcal M_{\varepsilon}$ is bounded in $W_{\varepsilon}$,
 then it cannot converge to zero in $L^{12/5}(\mathbb R^{3})$.
%
\end{lemma}
\begin{proof}
Otherwise by \eqref{propphivii} of Lemma \ref{lem:phi} we would have
$$1= \int_{\mathbb R^{3}} \phi_{u_{n}} u_{n}^{2}  \leq  |\phi_{u_{n}}|_{6} |u_{n}|^{2}_{12/5} \leq
C |u_{n}|^{2}_{12/5} =o_{n}(1)$$
which is a contradiction. 
\end{proof}

The unknown  $\lambda$ will be deduced as the Lagrange multiplier
associated to the critical point $u$ of $I_{\varepsilon}$ on $\mathcal M_{\varepsilon}$.
Indeed this is justified by the next  result.

\begin{lemma}\label{lem:variedade}
Under assumption \eqref{V} and \eqref{C}, the set $\mathcal M_{\varepsilon}$ is
bounded away from zero in the weak topology and 
is  weakly closed. Moreover it  is a $C^{1}$ manifold of codimension 1
homeomorphic to the unit sphere $\mathbb S_{\varepsilon}$ of $W_{\varepsilon}$.
\end{lemma}
\begin{proof}

If there is $\{u_{n}\}\subset \mathcal M_{\varepsilon}$ weakly convergent to $0$,
then, due to condition \eqref{C} we get a contradiction with Lemma \ref{lem:noLp}.


The fact it is weakly closed, follows again by condition \eqref{C} and Proposition \ref{prop:weaklyclosed}.

Since (see \cite{GP}) 
$$J'(u)[v]= \frac{1}{4}\int_{\mathbb R^{3}} \phi_{u} u v, \quad \forall u,v\in  W_{\varepsilon},$$ 
we see that whenever $u\in \mathcal M_{\varepsilon}$ the operator 
 $J'(u)$ is not the trivial one (since on the same $u$ gives $1/4$). Hence $\mathcal M_{\varepsilon}$
is a $C^{1}$ manifold of codiemension one.

To see that $\mathcal M_{\varepsilon}$ is homeomorphic to the unit sphere,
consider the {\sl projection map} 
$$\xi_{\varepsilon}:\mathbb S_{\varepsilon} \mapsto \mathcal M_{\varepsilon},\quad \text{ such that } \xi_{\varepsilon}(u)=t_{\varepsilon}(u)u$$
where $t_{\varepsilon}(u)$ is defined in \eqref{eq:te}.
Note that $\xi_{\varepsilon}$ is injective due to the unicity of $t_{\varepsilon}(u)$  and 
its easy to see that its inverse is the continuous {\sl retraction map}
$\xi_{\varepsilon}^{-1}(u) = u/\|u\|_{W_{\varepsilon}}$.

Moreover $\xi_{\varepsilon}$ is continuous. Actually
we show it is weakly continuous.
Let $\{u_{n},u\}\subset \mathbb S_{\varepsilon}$
with $u_{n}\rightharpoonup u$ in $W_{\varepsilon}$. 
In particular, by condition \eqref{C}
and Proposition \ref{prop:weaklyclosed} 
we infer
\begin{eqnarray}\label{eq:con}
\int_{\mathbb R^{3}} \phi_{u_{n}} u_{n}^{2} \to \int_{\mathbb R^{3}} \phi_{u} u^{2}. 
\end{eqnarray}
By using the H\"older inequality joint with \eqref{propphivii} of Lemma  \ref{lem:phi}
we have, for a suitable constant $C>0$,
\begin{eqnarray}\label{eq:1}
1=t_{\varepsilon}(u_{n})^{4} \int_{\mathbb R^{3}} \phi_{u_{n}} u_{n}^{2}\leq t_{\varepsilon}(u_{n})^{4}\|u_{n}\|_{W_{\varepsilon}}^{4}
\leq t_{\varepsilon}(u_{n})^{4} C
\end{eqnarray}
and we infer that $\{t_{\varepsilon}(u_{n})\}$  cannot tend to zero.
On the other hand, if $t_{\varepsilon}(u_{n})\to+\infty$, from the equality in \eqref{eq:1}
we deduce that 
$$\int_{\mathbb R^{3}} \phi_{u_{n}} u_{n}^{2}\to 0$$
and from \eqref{eq:con} $u=0$ which is a contradiction.

As a consequence, $t_{\varepsilon}(u_{n}) \to t\neq0$ (up to subsequence). Passing to the limit in  \eqref{eq:1}
we deduce that 
$$ 1= t^{4}\int_{\mathbb R^{3}} \phi_{u} u^{2}$$
which means that $t=t_{\varepsilon}(u)$ and
implies that $\xi_{\varepsilon}(u_{n})\to \xi_{\varepsilon}(u).$  This shows that $\xi_{\varepsilon}$ is an homeomorphism
concluding the proof.
\end{proof}

As by product of the proof of the proof of Lemma \ref{lem:variedade},
we state explicitly the following result  that will be useful later on.

\begin{corollary}\label{cor:continuidadet}
Under the assumptions and notation of Lemma \ref{lem:variedade}, if 
$\{u_{n},u\}\subset W_{\varepsilon}$ are such that
$u_{n}\rightharpoonup u\neq0$
in $H^{1}(\mathbb R^{3})$ then $t_{\varepsilon}(u_{n})\to t_{\varepsilon}(u)$.
In particular if $u\in \mathcal M_{\varepsilon}$ then $t_{\varepsilon}(u_{n}) 
\to 1$. 
\end{corollary}

We know that the functional $I_{\varepsilon}$ 
(under both assumptions \eqref{f_{1}} or \eqref{f_{1}}\,\hspace{-0,1cm}')
is positive and indeed we have
\begin{lemma}\label{lem:m}
Assume \eqref{V} and \eqref{C}. Then
$$\mathfrak m_{\varepsilon}:=\inf_{u\in \mathcal M_{\varepsilon}} I_{\varepsilon}(u)  >0.$$
\end{lemma}
\begin{proof}
If the infimum were zero,  then there would exists $\{u_{n}\}\subset \mathcal M_{\varepsilon}$ such that
$$\int_{\mathbb R^{3}} \phi_{u_{n}} u_{n}^{2} =1, 
\quad I_{\varepsilon}(u_{n}) = \frac{1}{2}\|u_{n}\|_{W_{\varepsilon}}^{2} +
 \int_{\mathbb R^{3}} F(u_{n})\to 0.$$
In particular $|u_{n}|_{12/5} \to 0$ contradicting Lemma \ref{lem:noLp}.
\end{proof}

Let us recall the notion of genus of Krasnoselky.
Given $A$ a closed and symmetric subset of some Banach space, with $0\notin A$, the {\sl genus }
of $A$, denoted as $\gamma(A)$, is defined as the least number $k\in \mathbb N$ such that there exists a continuous and even map $h:A\to \mathbb R^{k}\setminus\{0\}$. If such a map does not exists, the genus is set to $+\infty$
and finally $\gamma(\emptyset) = 0$. It is well known that the genus is a topological invariant
(under odd homeomorphism)
and that the genus of the  sphere in $\mathbb R^{N}$ is $N$,
while in infinite dimension it is $+\infty$.
Hence by Lemma \ref{lem:variedade} it follows that

\begin{corollary}\label{lem:genus}
Assume \eqref{V} and \eqref{C}.
Then the manifold $\mathcal M_{\varepsilon}$ 
(which is closed and symmetric with respect to the origin) has infinite genus.
\end{corollary}
\begin{proof}
Just observe that  $\mathcal M_{\varepsilon}$
is homeomorphic to the unit sphere via an odd homeomorphism.
\end{proof}

\medskip

Let us pass now to study the functional $I_{\varepsilon}$ defined in \eqref{eq:functional}, namely
\begin{equation*}
I_{\varepsilon}(u) = \frac{1}{2}\int_{\mathbb R^{3}}|\nabla u|^{2} + \int_{\mathbb R^{3}} V(\varepsilon x) u^{2} 
+ \int_{\mathbb R^{3}}F(u).
\end{equation*}
Here $\varepsilon>0$ is fixed.

\medskip

\noindent{\bf The compactness condition.} 
As it is standard in variational methods, we will need a compactness condition,
the so called {\sl Palais Smale condition}, that we recall here.
In general  given $I$, a $C^{1}$  functional on a Hilbert manifold $\mathcal M$,
a sequence $\{u_{n}\}\subset \mathcal M$ is said to be a {\sl Palais-Smale sequence} for $I$ 
(briefly, a $(PS)$ sequence) if $\{I(u_{n})\}$ is bounded and $I'(u_{n})\to 0$ in the  tangent bundle.
The functional, $I$ is said to satisfy the {\sl Palais-Smale condition} 
 if every $(PS)$ sequence has a convergent subsequence to an element of $\mathcal M$.

 The validity of this condition is strongly based on the compactness assumption \eqref{C}.

\begin{lemma}\label{lem:general}
Assume \eqref{f_{1}} (or \eqref{f_{1}}\,\hspace{-0,1cm}'), 
\eqref{f_{2}}, \eqref{f_{3}}, \eqref{V} and \eqref{C}. Then
the functional $I_{\varepsilon}$ satisfies the $(PS)$ condition on $\mathcal M_{\varepsilon}$.
\end{lemma}
\begin{proof}
Let $\{u_{n}\}\subset \mathcal M_{\varepsilon}$ be a $(PS)$ sequence for $I_{\varepsilon}$, then we can assume
\begin{eqnarray*}
I_{\varepsilon}(u_{n}) =\frac{1}{2}\int_{\mathbb R^{3}} |\nabla u_{n}|^{2} +\frac{1}{2}\int_{\mathbb R^{3}}V(\varepsilon x) u_{n}^{2} +\int_{\mathbb R^{3}} F(u_{n}) \to c
\end{eqnarray*}
and there  exists $\{\lambda_{n}\}\subset \mathbb R$ such that
\begin{equation}\label{eq:PSvincolada}
\forall v\in W_{\varepsilon}: 
\int_{\mathbb R^{3}} \nabla u_{n}\nabla v +\int_{\mathbb R^{3}} V(\varepsilon x)u_{n}v+\lambda_{n}\int_{\mathbb R^{3}}\phi_{u_{n}} u_{n} v +\int_{\mathbb R^{3}} f(u_{n})v =o_{n}(1).
\end{equation} 
Since  $I_{\varepsilon}$ is coercive, the sequence $\{u_{n}\}$ is bounded in $W_{\varepsilon}$,
then converges weakly to  $u$ 
and being $\mathcal M_{\varepsilon}$  weakly closed, we have
\begin{equation}\label{eq:uvinculo}
\int_{\mathbb R^{3}} \phi_{u}u^{2}=1.
\end{equation}
By  choosing $v=u_{n}$ in \eqref{eq:PSvincolada}, we have
\begin{equation}\label{eq:PSvincoladaun}
\int_{\mathbb R^{3}} |\nabla u_{n}|^{2} +\int_{\mathbb R^{3}} V(\varepsilon x)u_{n}^{2}+\lambda_{n}  +\int_{\mathbb R^{3}} f(u_{n})u_{n} =o_{n}(1)
\end{equation} 
and since $\{u_{n}\}$ is bounded in $W_{\varepsilon}$
we infer that  (using \eqref{eq:limitaf})
$$\left| \int_{\mathbb R^{3}} f(u_{n}) u_{n} \right|\leq 
 \delta |u_{n}|_{2}^{2} + C_{\delta}|u_{n}|_{q}^{q}\leq C.
$$
Then by  \eqref{eq:PSvincoladaun} we deduce that 
 $\{\lambda_{n}\}$ is bounded,  hence  converging, up to subsequences,  to some $\lambda.$


 By \eqref{eq:PSvincolada} again we have
$$
\forall v\in W_{\varepsilon}: 
\int_{\mathbb R^{3}} \nabla u \nabla v +\int_{\mathbb R^{3}}V(\varepsilon x) u v+\lambda\int_{\mathbb R^{3}}\phi_{u} u v +\int_{\mathbb R^{3}} f(u)v =0.
$$
In particular,  by taking $v=u$ we see that $\displaystyle \int_{\mathbb R^{3}}V(\varepsilon x) u^{2}<+\infty$
which joint to \eqref{eq:uvinculo} gives that $u\in \mathcal M_{\varepsilon}$.

Finally, by taking $v=u_{n}-u$ in \eqref{eq:PSvincolada} and passing to the limit, since
(as it is easy to see)
\begin{eqnarray*}\label{eq:}
\int_{\mathbb R^{3}} \phi_{u_{n}}u_{n} (u_{n}- u) \to 0\quad \text{and} \quad \int_{\mathbb R^{3}} f(u_{n})(u_{n}-u)\to0,
\end{eqnarray*}
we infer that $\|u_{n}\|_{W_{\varepsilon}} \to \|u\|_{W_{\varepsilon}}$.

Then $u_{n}\to u$ in $W_{\varepsilon}$ 
which concludes the proof.
%
%
\end{proof}

\section{Proof of Theorem \ref{th:signchanging} and Theorem \ref{th:negativo}
}\label{sec:genus}

As a consequence of the $(PS)$ condition we have  existence of ground state,
namely a minimizer for  $I_{\varepsilon}$ on $\mathcal M_{\varepsilon}$,
and actually infinitely many critical points under the oddness condition.

\medskip

\noindent{\bf Proof of Theorem \ref{th:signchanging}.}
The existence of the ground state is a consequence of the $(PS)$ condition. 
Of course 
$I_{\varepsilon}(\pm|u|)=I_{\varepsilon}(u)$ and we have actually a positive and a negative 
ground state.

Finally, by applying  the Krasnoselski Genus Theory we get the existence
of  infinitely many critical points $\{u_{n}\}$.
That  $\{u_{n}\}$ are at divergent critical levels 
follows from the abstract Theory.
Then  it is easy to see, since 
$$\int_{\mathbb R^{3}}F(u_{n}) \leq \int_{\mathbb R^{3}}\left( u_{n}^{2} +C |u_{n}|^{p} \right),$$
that $\{u_{n}\}$ are divergent also in norm. 
By noticing that  $f(t)t\geq0$,  the divergence of the Lagrange multipliers  follows.

\medskip

\medskip
\noindent{\bf Proof of Theorem \ref{th:negativo}.}
It follows by the $(PS)$ condition and the fact that $I_{\varepsilon}(-|u|)\leq I_{\varepsilon}(u).$

%

%
%

\begin{remark}\label{rem:bifurcation-e}
In case of negative ground states, since the functional essentially reduces to the (squared) norm
we can find an easy result concerning the bifurcation from the trivial solution $(0,\lambda)$
of the ground states.

To this aim, for any $\varepsilon,c>0$ let us denote with $u_{\varepsilon, c}$  the negative ground state solution
found in Theorem \ref{th:signchanging} or Theorem  \ref{th:negativo} (recall Remark \ref{rem:c}) on the constraint
$$\int_{\mathbb R^{3}} \phi_{u}u^{2}=c>0$$
and let $\lambda_{\varepsilon, c}$ be the associated Lagrange multiplier.
Then  explicitly
\begin{equation}\label{eq:bifurcation}
I_{\varepsilon}(u_{\varepsilon,c}) = \frac12 \|u_{\varepsilon,c}\|_{W_{\varepsilon}}^{2},
\quad \int_{\mathbb R^{3}} \phi_{u_{\varepsilon,c}} u_{\varepsilon,c}^{2} = c, 
\quad \lambda_{\varepsilon,c}  c = -\|u_{\varepsilon,c}\|_{W_{\varepsilon}}^{2}<0.
\end{equation}
We see that if $0<c_{1}<c_{2}$ then
\begin{eqnarray*}
\frac12\|u_{\varepsilon, c_{1}}\|_{W_{\epsilon}}^{2}= I_{\varepsilon}(u_{\varepsilon,c_{1}})  \leq I_{\varepsilon}\left( ({c_{1}}/{c_{2}})^{1/4} u_{\varepsilon,c_{2}}\right)
=\frac12(c_{1}/c_{2})^{1/2} \|u_{\varepsilon,c_{2}}\|_{W_{\varepsilon}}^{2}
\end{eqnarray*}
which means that the map
\begin{equation*}
c\in(0,+\infty) \mapsto \frac{\|u_{\varepsilon,c}\|^{2}_{W_{\varepsilon}}}{c^{1/2}} \in (0,+\infty) \quad
\text{is increasing}
\end{equation*}
and then
$$\exists \lim_{c\to0^+} \frac{\|u_{\varepsilon,c}\|^{2}_{W_{\varepsilon}}}{c^{1/2}} \in [0,+\infty).$$
In particular $ \lim_{c\to0^+} \|u_{\varepsilon,c}\|^{2}_{W_{\varepsilon}} =0.$
Consequently by \eqref{eq:bifurcation},
$$\lim_{c\to0^{+}}\lambda_{\varepsilon,c} c= - \lim_{c\to0^{+}} \|u_{\varepsilon, c}\|_{W_{\varepsilon}}^{2} = 0
$$
and we see that two cases hold:
\begin{itemize}
\item[(1)] there exists a sequence $c_n\to 0^+$ such that $\lim_{n\to+\infty} \lambda_{\varepsilon,c_n} = \overline \lambda \in(-\infty,0]$, or
\item[(2)] $\lim_{c\to0^{+}} \lambda_{\varepsilon,c} = -\infty$.
\end{itemize}
In the first case, then we have  a bifurcation point $(0,\overline \lambda)$;
in the second case then we have a bifurcation ``from $-\infty$''. 
\end{remark}



%

\section{The autonomous problem}\label{sec:autonomous}

In order to prove the multiplicity results involving condition \eqref{f_{1}},
it will be important  to consider the autonomous problem
associated to \eqref{eq:equazione}. 

\medskip

For a given constant potential
$\mu>0$ consider the problem
\begin{equation}\label{limite}\tag{$A_{\mu}$}
              -\Delta u + \mu u +\lambda\phi_{u} u+ f(u)=0
               \ \ \text{in} \ \ \mathbb R^3,  \medskip
\end{equation}
Let $H^{1}_{\mu}(\mathbb R^{3})$ be the usual subspace of $H^{1}(\mathbb R^{3})$
endowed with (squared) norm
$$\|u\|_{\mu}^{2} = \int_{\mathbb R^{3}} |\nabla u|^{2}+\mu \int_{\mathbb R^{3}} u^{2}.$$

The solutions $(u,\lambda)\in H_{\mu}^{1}(\mathbb R^{3}) \times \mathbb R$ of \eqref{limite}
are the critical points
of  the positive and $C^{1}$ functional 
\begin{equation*}
E_{\mu}(u)=\frac{1}{2}\int|\nabla  u|^2 +
          \frac{\mu}{2}\int u ^2 
          +\int_{\mathbb R^3}F(u) =\frac{1}{2}\|u\|_{\mu}^{2}+\int_{\mathbb R^3}F(u)
\end{equation*}
restricted to 
$$
\mathcal M_{\mu}= \left\{ u \in H_{\mu}^{1}(\mathbb R^{3}): J(u)=0 
\right\}, \quad J(u)= \displaystyle\int_{\mathbb R^{3}} \phi_{u} u^{2} - 1.
$$
It is clear that  $\mathcal M_{\mu}$ is not empty and has the same properties of 
$\mathcal M_{\varepsilon}$ given in  Lemma \ref{lem:noLp}, Lemma \ref{lem:variedade},
Corollary \ref{cor:continuidadet} and
Corollary \ref{lem:genus}. Finally, as in Lemma \ref{lem:m}, we have
$$\mathfrak m_{\mu}:=\inf_{u\in \mathcal M_{\mu}} E_{\mu}(u) >0.$$ 

Actually, in order to find solutions of \eqref{limite}, we work in the subspace of radial functions
since, by the Palais's Symmetric Criticality Principle, it is a natural constraint.
Then define $$\mathcal M_{\textrm{rad},\mu}:=\mathcal  M_{\mu} \cap H^{1}_{\text{rad}, \mu}(\mathbb R^{3})$$
(which evidently has the same properties of $\mathcal M_{\mu}$ and $\mathcal M_{\varepsilon}$) and
$$\mathfrak m_{\text{rad},\mu}:= \inf_{u\in \mathcal M_{\text{rad},\mu}} E_{\mu}(u) \geq\mathfrak m_{\mu}>0.$$

The advantage of the radial setting is that, due to the compact embedding of $H^{1}_{\textrm{rad},\mu}(\mathbb R^{3})$
into $L^{p}(\mathbb R^{3}),p\in(2,6)$, the manifold
$\mathcal  M_{\textrm{rad},\mu}$ is weakly closed.
Then we get the following compactness condition whose proof,
being very similar to that of Lemma \ref{lem:general}, is omitted.

\begin{lemma}\label{lem:PSE}
Assume
 \eqref{f_{1}} (or \eqref{f_{1}}\,\hspace{-0,1cm}'), 
\eqref{f_{2}} and \eqref{f_{3}}. Then the functional $E_{\mu}$ satisfies the $(PS)$ condition on $\mathcal  M_{\emph{rad},\mu}$.
\end{lemma}

Then we deduce a result analogous to Theorem \ref{th:signchanging} 
for critical points of $E_{\mu}$.
\begin{theorem}\label{th:gs-mu-infty}
Assume \eqref{f_{1}}\hspace{-0,07cm}', \eqref{f_{2}} and \eqref{f_{3}}. Then  any minimising sequence for $E_{\mu}$ on $\mathcal M_{\emph{rad},\mu}$ is convergent.
So $\mathfrak m_{\emph{rad},\mu}$ is achieved and the ground state
can be assumed of one sign.

Indeed the functional $E_{\mu}$ possesses infinitely many critical points
$\{u_{n}\}$ on $\mathcal M_{\emph{rad},\mu}$
with associated Lagrange multipliers  $\{\lambda_{n}\} \subset  (-\infty,0)$ satisfying
\begin{eqnarray*}
&&E_{\mu}(u_{n})=\frac{1}{2} \|u_{n}\|_{\mu}^{2}
+\int_{\mathbb R^{3}} F(u_{n})\to+\infty, \label{eq:HEnergyE}\\
&&\|u_{n}\|_{\mu}^{2}\to+\infty, \label{eq:HOscilE}\\
&&\lambda_{n} = - \left(\|u_{n}\|_{\mu}^{2}+\int_{\mathbb R^{3}} f(u_{n}) u_{n}  \right) \to -\infty. \label{eq:Enegativos}
\end{eqnarray*}
In particular \eqref{limite} has infinitely many solutions.
\end{theorem}

In case condition $\eqref{f_{1}}$ holds then we have
\begin{theorem}\label{th:gs-mu}
Assume \eqref{f_{1}}\,\hspace{-0,1cm}-\eqref{f_{3}}. Then any minimising sequence for $E_{\mu}$ on $\mathcal M_{\emph{rad},\mu}$ is convergent.
So  $\mathfrak m_{\emph{rad},\mu}$ is achieved on
a radial function, hereafter denoted with $\mathfrak u$, and moreover
$$\mathfrak m_{\emph{rad},\mu} = \mathfrak m_{\mu} = \min_{u\in \mathcal M_{\mu}} E_{\mu}(u) = E_{\mu}(\mathfrak u)>0$$
Finally  $\mathfrak u$ is negative and then $E_{\mu}(\mathfrak u)=\frac12\|\mathfrak u\|^{2}_{\mu}$.
\end{theorem}

We strength the fact that  $\mathfrak u$ has minimal energy on the whole $\mathcal M_{\mu}$,
namely, even between nonradial functions.
\begin{proof}
We need just to prove that $\mathfrak u$ realizes the minimum among all functions in $\mathcal M_{\mu}$
and  for this it is sufficient to show that for any $u\in \mathcal M_{\mu}$
there is another function in $\mathcal M_{\textrm{rad},\mu}$ with less energy.
 
Then let $u\in \mathcal M_{\mu}$,
  denote with $ u^{*}$ its Schwartz symmetrization and set 
$t_{*}>0$ such that $t_{*}u^{*} \in \mathcal M_{\mu}$. 
By the rearrangement inequality (see \cite[Theorem 3.7]{LL}) we get
$$\frac{1}{t_{*}^{4}}=\int_{\mathbb R^{3}} \phi_{ u^{*}}  (u^{*})^{2}
\geq \int_{\mathbb R^{3}} \phi_{ u}  u^{2} = 1$$
and deduce that $t_{*}\leq1$. Consequently, using  the properties of the spherical
rearrangement and that $f$ is positive, for a suitable $\xi\in (0,1)$:
\begin{eqnarray*}
E_{\mu}(t_{*}u^{*}) - E_{\mu}(u) &\leq& \frac12(t_{*}^{2}-1)\| u \|^{2}_{\mu} + \int_{\mathbb R^{3}}\left( F(t_{*}u) - F(u)\right)\\
&=&\frac12(t_{*}^{2}-1)\| u \|^{2}_{\mu} + (t_{*} -1)\int_{\mathbb R^{3}} f (\xi u) \\
&\leq&0,
\end{eqnarray*}
which concludes the proof.
The final part follows by  $F(-|u|)\leq F(u).$
\end{proof}

\begin{remark}
Analogously to Remark \ref{rem:bifurcation-e} we have bifurcation of the negative ground states found in Theorem \ref{th:gs-mu}
from the trivial solution  also for the autonomous problem \eqref{limite}.
\end{remark}

The ground state  $\mathfrak u$  found in  Theorem \ref{th:gs-mu} 
 will have a special role from now on.

We observe that all we have seen up to now
was valid for any fixed $\varepsilon>0$ and it was never used that the infimum $V_{0}$
of $V$ is achieved.

\section{The barycentre map and proof of Theorem \ref{th:LST} and Theorem \ref{th:MT}}
\label{Bary}

Without the oddness assumption of $f$ (namely  condition \eqref{f_{1}}\hspace{-0,02cm}'),
the multiplicity result is obtained thanks to the smallness of $\varepsilon$
and the fact that $V_{0}$ is achieved on a subset $M\subset\mathbb R^{3}$:

$$
0<\min_{x\in \mathbb R^3}V(x)=:V_{0}, \quad \text{with }\ 
M=\Big\{x\in\mathbb R^{3}:V(x)=V_{0}  \Big\}.
$$
Without  without loss of generality, we assume $0\in M$.
Define the set of  
negative functions:
 $$
 N:=\left\{ u:\mathbb R^{3} \to (-\infty,0] \right\}.$$

Consider the autonomous problem
$$-\Delta u+V_{0} u +\lambda \phi_{u} u +f(u) = 0 \quad \text{in }\mathbb R^{3}$$
and let  $\mathfrak u$ be  the radial and negative function satisfying
$$\mathfrak m_{V_{0}}= \min_{u\in \mathcal M_{V_{0}}} E_{V_{0}}(u)=E_{V_{0}}(\mathfrak u) >0,$$
  see Theorem  \ref{th:gs-mu}.

 Finally, since $\mathcal M_{\varepsilon}\subset\mathcal M_{V_{0}}$
 and $V(x)\ge V_{0}$, it is
 $$E_{V_{0}}(\mathfrak u) \leq \mathfrak m_{\varepsilon}.$$
 
\medskip

For $T > 0$ define  
 $\eta$ the  smooth nonincreasing
cut-off function defined in $[0,\infty)$ by
\begin{equation*}\label{eta}
\eta(s)=
\begin{cases}
1 & \mbox{ if } 0 \leq s \leq T/2\\
0 & \mbox{ if } s\geq T
\end{cases}
\end{equation*}
and for any $y \in M$, set
\begin{eqnarray*}
\Psi_{\varepsilon , y }(x) :=\eta (|\varepsilon x - y|)\mathfrak u\biggl(\frac{\varepsilon x -y}{\varepsilon}\biggl).
\end{eqnarray*}
Let $t_{\varepsilon,y}:=t_{\varepsilon}(\Psi_{\varepsilon,y})>0$ 
such that
$t_{\varepsilon,y}\Psi_{\varepsilon,y}\in \mathcal M_{\varepsilon}$,
and define the map
\begin{equation*}\label{Phi}
\Phi_{\varepsilon}: y\in M\mapsto  t_{\varepsilon,y}\Psi_{\varepsilon ,y}\in \mathcal M_{\varepsilon}
\end{equation*}
which is easily seen to be continuous.
By construction, for any
$y\in M$,    $\Phi_{\varepsilon}(y)$  has compact support
and $\Phi_{\varepsilon}(y)\in \mathcal M_{\varepsilon} \cap N$.
In particular
$$I_{\varepsilon}(\Phi_{\varepsilon}(y)) =\frac12 \|\Phi_{\varepsilon}(y)\|_{W_{\varepsilon}}^{2}.$$

\medskip

\begin{lemma}\label{gio14}
Assume  \eqref{f_{1}}-\eqref{f_{3}}, \eqref{V1} and \eqref{C}. Then
$$\lim_{\varepsilon \rightarrow 0^{+}}I_{\varepsilon}(\Phi_{\varepsilon}(
y))=E_{V_{0}} (\mathfrak u), \quad \text{uniformly} \ in \ y \in M.
$$

\end{lemma}
\begin{proof}
Suppose by contradiction that  there exist $\delta_{0}> 0$, $\varepsilon_{n}\rightarrow 0^{+}$
and  $\{y_{n}\} \subset M$
 such that
\begin{eqnarray}\label{1eq7}
| I_{\varepsilon_{n}}(\Phi_{\varepsilon_{n}}(y_{n}))- E_{V_{0}}(\mathfrak u)|
\geq \delta _{0}.
\end{eqnarray}

From the Lebesgue's Theorem, we deduce
\begin{eqnarray}\label{eq:um}
\lim_{n\rightarrow\infty} 
\int_{\mathbb R^{3}} |\nabla \Psi_{\varepsilon_{n},y_{n}} |^{2} =\int_{\mathbb R^{3}} |\nabla \mathfrak u|^{2},
\quad 
\lim_{n\rightarrow\infty} \int_{\mathbb R^{3}} V(\varepsilon_{n} x) 
\Psi_{\varepsilon_{n},y_{n}}^{2}
=V_{0}\int_{\mathbb R^{3}} \mathfrak u^{2}.
\end{eqnarray}
In particular 
$ \{\Psi_{\varepsilon_{n},y_{n}}\}$ is bounded in $W_{\varepsilon_{n}}$,
and so  weakly convergent to some $v\in W_{\varepsilon}$ 
and a.e. in $\mathbb R^{3}$. By  \eqref{eq:um} it has to be  $v=\mathfrak u$,
and therefore we have, due to the compactness assumption \eqref{C},
$$\Psi_{\varepsilon_{n},y_{n}}\rightharpoonup  \mathfrak u \quad \text{ in } W_{\varepsilon}
, \quad
\Psi_{\varepsilon_{n},y_{n}}\to \mathfrak u \quad \text{ in } L^{p}(\mathbb R^{3}), \ p\in (2,6).$$

Recalling  that
$\Phi_{\varepsilon_{n}}(y_{n})=  t_{\varepsilon_{n},y_{n}}\Psi_{\varepsilon_{n} ,y_{n}}\in \mathcal M_{\varepsilon_{n}}$
and Proposition \ref{prop:weaklyclosed} we get
$$\frac{1}{t_{\varepsilon_{n}, y_{n}}^{4}} =\int_{\mathbb R^{3}} \phi_{\Psi_{\varepsilon_{n}, y_{n}}}\Psi_{\varepsilon_{n}, y_{n}}^{2} = \int_{\mathbb R^{3}} \phi_{\mathfrak u}\mathfrak u^{2} +o_{n}(1) = 1+o_{n}(1)
$$
which implies that $t_{\varepsilon_{n},y_{n}}\to1$. But then using
\eqref{eq:um} 
we conclude that
\begin{eqnarray*}
I_{\varepsilon_{n}}(t_{\varepsilon_{n},y_{n}} \Psi_{\varepsilon_{n},  y_{n}})  &= &\frac{t_{\varepsilon_{n},y_{n}}^{2}}{2}\int_{\mathbb R^{3}}
|\nabla \Psi_{\varepsilon_{n}, y_{n}}|^{2} +  \frac{t_{\varepsilon_{n},y_{n}}^{2}}{2}
\int_{\mathbb R^{3}} V(\varepsilon_{n} x) 
\Psi_{\varepsilon_{n}, y_{n}}^{2}
\\
&\to& E_{V_{0}}(\mathfrak u).
\end{eqnarray*}
contradicting \eqref{1eq7}.
\end{proof}

\bigskip

By Lemma \ref{gio14}, 
$h(\varepsilon):=|I_{\varepsilon}(\Phi_{\varepsilon}(y))-E_{V_{0}}(\mathfrak u)|=o_{\varepsilon}(1)$
for $\varepsilon\to 0^{+}$ uniformly in $y$,
and then 
$$I_{\varepsilon}(\Phi_{\varepsilon}(y))-E_{V_{0}}(\mathfrak u)\leq 
 \big|I_{\varepsilon}(\Phi_{\varepsilon}(y))-E_{V_{0}}(\mathfrak u)\Big|\leq h(\varepsilon)
=o_{\varepsilon}(1).$$
In particular the sub level set
\begin{equation}\label{subnehari}
{\cal{M_{\varepsilon}}}^{E_{V_{0}}(\mathfrak u)+h(\varepsilon)}:=\Big\{u\in{\cal{M_{\varepsilon}}}:I_{\varepsilon}(u)\leq
E_{V_{0}}(\mathfrak u)+h(\varepsilon)\Big\}
\end{equation}
is not empty, since 
 for sufficiently small $\varepsilon$,
\begin{equation}\label{eq:MNP}
 \Phi_{\varepsilon}(y)\in {\cal{M_{\varepsilon}}}^{E_{V_{0}}(\mathfrak u)+h(\varepsilon)}\cap N.
\end{equation}

\bigskip

\noindent{\bf The barycentre map.} 
We are in a position now to define the barycenter map
that will send a convenient sublevel in $\mathcal M_{\varepsilon}$
in a suitable neighborhood of $M$.
From now on we fix a  $T >0$ in such a way that
$M$ and 
$$M_{2T}:=\Big\{x\in \mathbb R^{3}: d(x,M)\leq 2T\Big\}$$ are homotopically equivalent
($d$ denotes the euclidean distance).
In particular they are also homotopically equivalent to 
$$M_{T}:=\Big\{x\in \mathbb R^{3}: d(x,M)\leq T\Big\}.$$
Let $\rho=\rho(T) > 0$ be such that
$M_{2T}\subset B_{\rho}$ and  $\chi : \mathbb R^{3}\rightarrow
\mathbb R^{3}$ be defined as
\begin{equation*}\label{chi}
\chi(x)= 
\begin{cases}
x & \mbox{ if } | x | \leq \rho\\
\rho \dis\frac{x}{|x|} & \mbox{ if } | x | \geq \rho.
\end{cases}
\end{equation*}
 Finally, let the {\sl barycenter map}
$\beta_{\varepsilon}$ defined on functions with compact support $u\in W_{\varepsilon}$ by 
$$
\beta_{\varepsilon}(u):=\frac{\dis\int_{\mathbb R^{3}}\chi(\varepsilon x)
u^{2}}{\dis\int_{\mathbb R^{3}}u^{2}}\in \mathbb R^{3}.
$$

The next three lemmas give the behaviour of $\beta_{\varepsilon}$
and $I_{\varepsilon}$.

\begin{lemma}\label{gio15} 
Under assumption \eqref{V1}, the function $\beta_{\varepsilon}$ satisfies
\begin{eqnarray*}
\lim_{\varepsilon \rightarrow 0^{+}}\beta_{\varepsilon}(\Phi_{\varepsilon}(y)) = y, \ \
\text{uniformly in }  y \in M.
\end{eqnarray*}
\end{lemma}
\begin{proof} Suppose, by contradiction, that the lemma is
false. Then, there exist $\delta_{0}> 0$, $\varepsilon_{n}\rightarrow 0^{+}$ and $\{y_{n}\} \subset M$
 such that
\begin{eqnarray}\label{1eq11}
| \beta_{\varepsilon_{n}}(\Phi_{\varepsilon_{n}}(y_{n})) - y_{n}| \geq
\delta_{0}.
\end{eqnarray}
Using the definition of $\Phi_{\varepsilon_{n}}(y_{n}),\beta_{\varepsilon_{n}}$ and $\eta$ given above,
we have the  equality
\begin{eqnarray*}
\beta_{\varepsilon_{n}}(\Phi_{\varepsilon_{n}}(y_{n}))= y_{n} + \frac{\dis\int_{\mathbb R^{3}}
[\chi(\varepsilon_{n}z +
y_{n})-y_{n}]\Big|\eta(|\varepsilon_{n}z|)\mathfrak u(z)\Big|^{2}}
{\dis \int_{\mathbb R^{3}}\Big|\eta(|\varepsilon_{n}z|)\mathfrak u(z)\Big|^{2}}.
\end{eqnarray*}
Using the fact that $\{y_{n}\}\subset M\subset B_{\rho} $ and the
Lebesgue's Theorem, it follows
\begin{eqnarray*}
| \beta_{\varepsilon_{n}}(\Phi_{\varepsilon_{n}}(y_{n})) - y_{n}| = o_{n}(1),
\end{eqnarray*}
which contradicts (\ref{1eq11}) and the Lemma is proved.
\end{proof}

\begin{lemma}\label{lem:convground}
Assume
 \eqref{f_{1}}-\eqref{f_{3}}, \eqref{V1} and \eqref{C}.
If $\varepsilon_{n}\rightarrow 0$ and $\{u_{n}\}\subset \mathcal M_{\varepsilon_{n}}$ is such that
$I_{\varepsilon_{n}}(u_{n})\rightarrow   E_{V_{0}}(\mathfrak u)$,  then 
$\{u_{n}\}$ converges to $\mathfrak u$ in $H^{1}_{V_{0}}(\mathbb R^{3})$.

Then, for $n$ sufficiently large  $\{u_{n}\}$ can be assumed negative.

\end{lemma}
\begin{proof}
Since $\{u_{n}\}\subset \mathcal M_{\varepsilon_{n}}\subset \mathcal M_{V_{0}}$
$$| I_{\varepsilon_{n}}(u_{n})- E_{V_{0}} (u_{n}) |\leq
 \int_{\mathbb R^{3}} \left(V(\varepsilon_{n} x) - V_{0}\right) u_{n}^{2} \to 0
$$
we deduce that $E_{V_{0}}(u_{n}) \to E_{V_{0}}(\mathfrak u)$, namely $\{u_{n}\}$
is a minimising sequence for $E_{V_{0}}$ on $\mathcal M_{V_{0}}$.
The result follows by Theorem \ref{th:gs-mu}.
\end{proof}

\begin{lemma}\label{lem:finalconvergence}
Assume  \eqref{f_{1}}-\eqref{f_{3}}, \eqref{V1} and \eqref{C}.
Then
$$\lim_{\varepsilon \rightarrow 0^{+}}\ \sup_{u \in
\cal{M_{\varepsilon}}^{E_{V_{0}}(\mathfrak u)+h(\varepsilon)}\cap N} \ \inf_{y \in
M_{T}}\Big| \beta_{\varepsilon} (u)- y \Big|  = 0;
$$
\end{lemma}
\begin{proof}Let $\{\varepsilon_{n}\}$ be such that
$\varepsilon_{n}\rightarrow 0^{+}$. For each $n \in \mathbb N$, there exists
$u_{n}\in  \cal{M}_{\varepsilon_{n}}^{E_{V_{0}}(\mathfrak u)+h(\varepsilon_{n})}\cap N$ such that 
$$
\inf_{y \in M_{T}}\Big| \beta_{\varepsilon_{n}} (u_{n})- y \Big|
= \sup_{u \in \cal{M}_{\varepsilon_{n}}^{E_{V_{0}}(\mathfrak u)+h(\varepsilon_{n})} \cap N} \inf_{y
\in M_{T}}\Big|\beta_{\varepsilon_{n}} (u)- y \Big|  +  o_{n}(1).
$$
Thus, it suffices to find a sequence $\{y_{n}\}\subset M_{T}$
such that
\begin{eqnarray}\label{1eq18}
\lim_{n\rightarrow  \infty}\biggl|\beta_{\varepsilon_{n}}(u_{n})-y_{n}\biggl| \ \ = 0.
\end{eqnarray}
Actually this holds for {\sl any} sequence $\{y_{n}\}\subset M_{T}$.
Indeed  since $\{u_{n}\}\subset  \mathcal M_{V_{0}}$
(and since
under assumption \eqref{f_{1}}, $\mathfrak u$
is the ground state of $E_{V_{0}}$ on the whole $\mathcal M_{V_{0}}$), 
 we have
\begin{eqnarray*}\label{1eq19}
 E_{V_{0}} (\mathfrak u) \leq E_{V_{0}} (u_{n}) 
 \leq I_{\varepsilon_{n}}(u_{n})\leq
E_{V_{0}}(\mathfrak u)+ h(\varepsilon_{n})
\end{eqnarray*}
which implies that 
$ I_{\varepsilon_{n}}(u_{n})\rightarrow E_{V_{0}}(\mathfrak u) .
$
Then by Lemma \ref{lem:convground},  
\begin{equation}\label{eq:convu}
\{u_{n}\}  \ \text{ is convergent in to $\mathfrak u$ in }  H^{1}_{V_{0}}(\mathbb R^{3}).
\end{equation}
Then if $\{y_{n}\}$ is any sequence in $M_{T}$, since
$$
\beta_{\varepsilon_{n}}(u_{n})= y_{n}+\frac{\dis\int_{\mathbb R^{3}}[\chi (\varepsilon_{n}z +
y_{n})- y_{n}] u_{n}^{2}(z)}{\dis\int_{\mathbb R^{3}}
u_{n}(z)^{2}},
$$
by using \eqref{eq:convu}
we see that $\{y_{n}\}$ verifies (\ref{1eq18}).
\end{proof}

\medskip

In virtue of  Lemma \ref{lem:finalconvergence},
 there exists $\varepsilon^{*}>0$ such that 
\begin{eqnarray*} 
&&\sup_{u\in {\mathcal M_{\varepsilon}}^{E_{V_{0}}(\mathfrak u)+h(\varepsilon)}  \cap N } d(\beta_{\varepsilon}(u), M_{T})<T/2, \label{eq:supN}.
\end{eqnarray*}

Define now
$$M^{+}:=M_{3T/2}=\Big\{x\in \mathbb R^{3}: d(x,M)\leq 3T/2\Big\}$$
so that  $M$ and $M^{+}$ are homotopically equivalent.

Now,  reducing  $\varepsilon^{*}>0$ if necessary,
we can assume that Lemma \ref{gio15},  Lemma \ref{lem:finalconvergence} and 
\eqref{eq:MNP} holds.
Then by standard arguments (see e.g. \cite{BC,BCP}) the composed map
\begin{equation}\label{eq:fundamentalN}
M\stackrel{\Phi_{\varepsilon}}{\longrightarrow} {\cal{M}_{\varepsilon}}^{E_{V_{0}}(\mathfrak u)+h(\varepsilon) }
\cap N\stackrel{\beta_{\varepsilon}}{\longrightarrow}M^{+}
\quad \text{ is homotopic to the inclusion map}
\end{equation}

At this point we can finish the proof of the multiplicity result by implementing the Ljusternick-Schnirelmann Theory.


\medskip

\noindent{\bf The Ljusternick-Schnirelmann category: proof of Theorem \ref{th:LST}.}
By \eqref{eq:fundamentalN} and the very well known properties of the category,
we get, for  any $\varepsilon\in (0,\varepsilon^{*}]$,
\begin{eqnarray*}
\cat({\cal{M}_{\varepsilon}}^{E_{V_{0}}(\mathfrak u)+h(\varepsilon)} \cap N)\geq
\cat_{M^{+}}(M).
\end{eqnarray*}
Then, since the $(PS)$ condition holds (Lemma \ref{lem:general}), the Ljusternik-Schnirelman Theory  (see e.g.
\cite{Sz}) applies and
$I_{\varepsilon}$ has at least $\cat_{M^{+}}(M)=\cat(M)$ critical points
on ${\cal{M}_{\varepsilon}}$ with energy less then $E_{V_{0}}(\mathfrak u)+h(\varepsilon)$;
so we have found  $\cat(M)$  solutions for problem \eqref{eq:equazione}
which are negative.


\medskip

To find the other solution we argue as in \cite{BCP}.
Since $M$ is not contractible, 
the  compact set $\mathcal A:=\overline{\Phi_{\varepsilon}(M)}$
can not be contractible in ${\cal{M}_{\varepsilon}}^{E_{V_{0}}(\mathfrak u)+h(\varepsilon)}$. 
Moreover we can choose $v\leq0, v\in\mathcal M_{\varepsilon}\setminus \mathcal A$ 
so $v$ can not be multiple of any element of $\mathcal A$.
In particular   
$I_{\varepsilon}(v)> E_{V_{0}}(\mathfrak u)+h(\varepsilon).$

Let 
$$\mathfrak C:=\Big\{tv +(1-t)u: t\in [0,1], u\in\mathcal A \Big\}$$
be the cone (hence compact and contractible) generated by $v$ over $\mathcal A$.
It follows that  $0\notin \mathfrak C$.

Consider   also (see the map defined in the proof of Lemma \ref{lem:variedade})
$$\xi_{\varepsilon}(\mathfrak C)=\Big\{t_{\varepsilon}(w)w: w\in \mathfrak C\Big\}$$
 the projection of the cone on $\mathcal M_{\varepsilon}$, compact  as well, and define
$$c:=\max_{t_{\varepsilon}(\mathfrak C)}I_{\varepsilon}
>E_{V_{0}}(\mathfrak u)+h(\varepsilon).$$
Since $\mathcal A\subset \xi_{\varepsilon}(\mathfrak C)\subset \mathcal M_{\varepsilon}$
and $\xi_{\varepsilon}(\mathfrak C)$ is contractible in $\mathcal M^{c}_{\varepsilon}:=\{u\in \mathcal M_{\varepsilon}: I_{\varepsilon}(u)\leq c\}$,
it follows that also $\mathcal A$ is contractible in $\mathcal M_{\varepsilon}^{c}$.

Summing up, we have a set $\mathcal A$ which is contractible in $\mathcal M^{c}_{\varepsilon}$
but not  in $\mathcal M_{\varepsilon}^{E_{V_{0}}(\mathfrak u)+h(\varepsilon)}$, and $c>E_{V_{0}}(\mathfrak u)+h(\varepsilon).$ The reason of that, since $I_{\varepsilon}$ satisfies the $(PS)$
condition, is due to the existence of another critical level 
between $E_{V_{0}}(\mathfrak u)+h(\varepsilon)$ and $c$. Then we have another critical
point in $\mathcal M_{\varepsilon}\cap N$ with higher energy.

%
The proof of Theorem \ref{th:LST} is thereby complete.

\medskip

\noindent {\bf The Morse Theory: proof of Theorem \ref{th:MT}.} 
Here we prove  Theorem \ref{th:MT} hence assumptions 
\eqref{fC1}-\eqref{f_{5}} as well as \eqref{V1} and \eqref{C}
are  assumed here once for all.

\medskip

Let us recall first few basic definitions and fix some notations.

Given a pair  $(X,Y)$ of topological spaces with
$Y\subset X,$ let $H_{*}(X,Y)$ be its singular homology with coefficients in some field $\mathbb F$
(from now on omitted) and 
$$\mathcal P_{t}(X,Y)=\sum_{k}\dim H_{k}(X,Y)t^{k}$$
its Poincar\'e polynomial. Whenever $Y=\emptyset$, then it will be always omitted in 
all the objects which involve the pair.

Recall also that  if $H$ is an Hilbert space,  $I:H\to \mathbb R$  a $C^{2}$ functional  and 
$u$ an isolated critical point  with $I(u)=c$, the  {\sl polynomial Morse index} of $u$ is
defined as
$$\mathcal I_{t}(u)=\sum_{k}\dim C_{k}(I,u)  t^{k}.$$
Here, given the sublevel $I^{c}=\{u\in H: I(u)\leq c\}$ and a neighborhood 
$U$ of the critical point $u$, 
 $C_{k}(I,u)=H_{k}(I^{c}\cap U, (I^{c}\setminus\{u\})\cap U)$ denote the critical groups.
The multiplicity of $u$ is the number $\mathcal I_{1}(u)$.

When $I''(u)$ is associated to a selfadjoint isomorphism,
then the critical point $u$ is said to be non-degenerate and 
it holds $\mathcal I_{t}(u)=t^{ m (u)}$,
where $ m(u)$ is the {\sl (numerical) Morse index of $u$}: the maximal dimension
of the subspaces where $I''(u)[\cdot,\cdot]$ is negative definite.

%
%

\bigskip

\begin{lemma}
The functional $I_{\varepsilon}$ is of class $C^{2}$ and
for $u,v,w\in W_{\varepsilon}$
\begin{equation*}\label{I''}
I_{\varepsilon}''(u)[v,w]=\int_{\mathbb R^{3}}\nabla v \nabla w+\int_{\mathbb R^{3}}V(\varepsilon x)vw+\int_{\mathbb R^{3}}f'(u)vw.
\end{equation*}
Moreover  $I_{\varepsilon}''(u)$ is represented by the operator
\begin{equation*}
\emph L_{\varepsilon}(u):=\emph R(u)+\emph K(u): W_{\varepsilon}\to W_{\varepsilon}',
\end{equation*}
where $\mathrm R(u)$ is the Riesz isomorphism  and $\emph{K}(u)$ is compact.
\end{lemma}
\begin{proof}
By \eqref{eq:limitaf'} $I_{\varepsilon}''$ is well defined and continuous.
Then 
\begin{equation*}
I_{\varepsilon}(u)\approx\textrm L_{\varepsilon}(u):=\textrm R(u)+\mathrm K(u): W_{\varepsilon}\to W_{\varepsilon}'.
\end{equation*}
Let us show that, for $u\in W_\varepsilon$,  $\mathrm{K}(u)$ is compact. Let then 
$v_{n}\rightharpoonup 0$ in $W_{\varepsilon}$ and $w\in W_{\varepsilon}$.
By \eqref{eq:limitaf'} we get 
that given $\delta > 0$ for some constant $C_{\delta}>0$:
$$
\int_{\mathbb R^{3}}\left|f'(u)v_{n}w\right| 
\leq \delta \left |v_{n}\right|_{2} \left|w\right|_{2} + C_{\delta}\left|u\right|^{q-2}_{q}\left|v_{n}\right|_{q}\left |w\right |_{q}
$$
and the last term tends to zero due to assumption \eqref{C}.
%
%
%
%
By the arbitrarily of  $\delta$, we  deduce
$$\|\mathrm K(u)[v_{n}]\|=\sup_{\|w\|_{W_{\varepsilon}}=1}\Big|\int_{\mathbb R^{3}}f'(u)v_{n}w\Big|\rightarrow 0,$$
namely the compactness of $\textrm{K}(u)$.
\end{proof}

Now  for $a\in(0,+\infty]$, define the sublevels of the functional
$$I_{\varepsilon}^{a}:=\Big\{u\in W_{\varepsilon}: I_{\varepsilon}(u)\leq a\Big\}\ , \qquad 
\mathcal M_{\varepsilon}^{a}:= \mathcal M_{\varepsilon}\cap  I_{\varepsilon}^{a}$$
and the sets of critical points
$$\mathcal K_{\varepsilon}:=\Big\{u\in W_{\varepsilon}: I'_{\varepsilon}(u)=0\Big\}\ ,\qquad
\mathcal K_{\varepsilon}^{a}:= \mathcal K_{\varepsilon}\cap  I_{\varepsilon}^{a}\ ,\qquad
(\mathcal K_{\varepsilon})_{a}:=\Big\{u\in \mathcal K_{\varepsilon}: I_{\varepsilon}(u)> a\Big\} .
$$
In the remaining part of this section we will follow \cite{BC,claudianorserginhorodrigo}.

Let $\varepsilon^{*}>0$ small as at the end of Section \ref{Bary} 
and let $\varepsilon\in (0,\varepsilon^{*}]$ be fixed.
In particular $I_{\varepsilon}$ satisfies the Palais-Smale condition.
 We are going to prove that $I_{\varepsilon}$ restricted to $\mathcal M_{\varepsilon}$
 has at least $2\mathcal P_{1}(M)-1$ critical points.
 \medskip
 
 We can assume,
 of course,  that there exists a regular value $b^{*}_{\varepsilon}>E_{V_{0}}(\mathfrak u)$
 for the functional $I_{\varepsilon}$. Moreover, possibly reducing $\varepsilon^{*}$, we can assume that, see \eqref{subnehari},
 $$\Phi_{\varepsilon}: M\to  \mathcal M^{E_{V_{0}}(\mathfrak u)+h(\varepsilon)}_{\varepsilon}\cap N\subset \mathcal M_{\varepsilon}^{b_{\varepsilon}^{*}}.$$

Since $\Phi_{\varepsilon}$ is injective, it induces
injective homomorphisms in the homology groups, then $\dim H_{k}(M)
\leq \dim H_{k}(\mathcal M_{\varepsilon}^{b_{\varepsilon}^{*}})$
and consequently
\begin{equation}\label{obvious1}
\mathcal P_{t}(\mathcal M_{\varepsilon}^{b_{\varepsilon}^{*}})=\mathcal P_{t}(M)+\mathcal Q(t), \qquad \mathcal Q\in \mathbb P,
\end{equation}
where  $\mathbb P$ is the set of all polynomials with non-negative integer coefficients.
 
 \medskip

As in  \cite[Lemma 5.2]{BC} we have
\begin{lemma}
Let  $r\in (0,E_{V_{0}}(\mathfrak u) )$  and $a\in(r,+\infty]$ a regular level for $I_{\varepsilon}$.
Then
 \begin{eqnarray}\label{PtP}
\mathcal P_{t}(I_{\varepsilon}^{a},I_{\varepsilon}^{r})&=&t \mathcal P_{t}(\mathcal M_{\varepsilon}^{a}).
\end{eqnarray}
\end{lemma}

\medskip

 Then following result holds.
\begin{corollary}\label{t}
Let $r\in (0,\mathfrak m_{V_{0}})$. Then
\begin{eqnarray*}
\mathcal P_{t}(I_{\varepsilon}^{b_{\varepsilon}^{*}},I_{\varepsilon}^{r})&=&t\Big(\mathcal P_{t}(M)+\mathcal Q(t)\Big), \qquad \mathcal Q\in \mathbb P, \label{prima}\\
\mathcal P_{t}(W_{\varepsilon}, I_{\varepsilon}^{r})&=& 
t.\label{seconda}
\end{eqnarray*}
\end{corollary}
\begin{proof}
The first equality follows by    \eqref{obvious1} and \eqref{PtP} simply by choosing  $a=b^{*}_{\varepsilon}$.
The second one follows by \eqref{PtP} with $a=+\infty$ and
recalling that  $\mathcal M_{\varepsilon}$ is contractible. 
\end{proof}

To deal with critical points above the regular level $b^{*}_{\varepsilon},$ we recall  also the following
result whose proof is only based on notions of algebraic topology 
and is exactly as in \cite[Lemma 5.6]{BC}, see also \cite[Lemma 2.4]{claudianorserginhorodrigo}.

\begin{lemma}\label{quadrato}
It holds
$$\mathcal P_{t}(W_{\varepsilon},I_{\varepsilon}^{b_{\varepsilon}^{*}})=t^{2}\Big(\mathcal P_{t}(M)+\mathcal Q(t)-1\Big), \qquad \mathcal Q\in \mathbb P.$$
\end{lemma}

Then by using the Morse Theory we arrive at the following fundamental result.

\begin{corollary}\label{separato}
Suppose that 
the set $\mathcal K_{\varepsilon}$
is discrete. Then
\begin{eqnarray*}\label{1}
\sum_{u\in \mathcal {K_{\varepsilon}}^{b^{*}_{\varepsilon}}}\mathcal I_{t}(u)=t\Big(\mathcal P_{t}(M)+\mathcal Q(t)\Big)+(1+t)\mathcal Q_{1}(t)
\end{eqnarray*}
and 
\begin{eqnarray*}
\sum_{u\in(\mathcal K_{\varepsilon})_{b^{*}_{\varepsilon}}} \mathcal I_{t}(u)=t^{2}\Big(\mathcal P_{t}(M)+\mathcal Q(t)-1\Big)+(1+t)\mathcal Q_{2}(t),
\end{eqnarray*}
where $\mathcal Q,\mathcal Q_{1}, \mathcal Q_{2}\in \mathbb P.$

\end{corollary}
\begin{proof}
Indeed the Morse theory 
gives
\begin{eqnarray*}\label{1}
\sum_{u\in \mathcal K^{b_{\varepsilon}^{*}}_{\varepsilon}}\mathcal I_{t}(u)=\mathcal P_{t}(I_{\varepsilon}^{b_{\varepsilon}^{*}}, I_{\varepsilon}^{r})+(1+t)\mathcal Q_{1}(t)\\
\end{eqnarray*}
and 
\begin{eqnarray*}
\sum_{u\in(\mathcal K_{\varepsilon})_{b_{\varepsilon}^{*}}} \mathcal I_{t}(u)=
\mathcal P_{t}(W_{\varepsilon},I_{\varepsilon}^{b_{\varepsilon}^{*}})+(1+t)\mathcal Q_{2}(t)
\end{eqnarray*}
so that, by using  Corollary \ref{t} and Lemma \ref{quadrato}, we easily conclude
\end{proof}

Then by Corollary \ref{separato} we get
$$\sum_{u\in \mathcal K_{\varepsilon}}\mathcal I_{t}(u)=t\mathcal P_{t}(M)+t^{2}\Big(\mathcal P_{t}(M)-1\Big)+t(1+t)\mathcal Q(t)$$
for some $\mathcal Q\in \mathbb P.$ 
We easily deduce that, if the critical points of $I_{\varepsilon}$
are non-degenerate, then they are  at least $2\mathcal P_{1}(M)-1$,
if counted with their multiplicity.

Then the proof of Theorem \ref{th:MT} is  complete.
\medskip

\subsection*{Disclosure of potential conflict of interests}
The authors do not have any conflicts of interest to declare.

\end{document}